\documentclass[a4paper, 12pt]{article}

\usepackage[noblocks]{authblk}
\usepackage{a4wide}
\usepackage{amsmath, amssymb, amsthm}
\usepackage{pstricks}
\usepackage{enumerate}

\title{Rainbow Connection Number and \\ Connected Dominating Sets}
\author[1]{L. Sunil Chandran}
\author[1]{Anita Das}
\author[1]{Deepak Rajendraprasad}
\author[2]{Nithin M. Varma}
\affil[1]
{
	Department of Computer Science and Automation, \authorcr 
	Indian Institute of Science, \authorcr
	Bangalore -560012, India. \authorcr
	\{sunil, anita, deepakr\}@csa.iisc.ernet.in
}
\affil[2]
{
	Department of Computer Science and Engineering, \authorcr
	National Institute of Technology, \authorcr
	Calicut - 673 601, India. \authorcr
	nithvarma@gmail.com
}
\date{}

\theoremstyle{definition}
\newtheorem{defn}{Definition}

\theoremstyle{plain}
\newtheorem{thm}{Theorem}
\newtheorem{lem}[thm]{Lemma}
\newtheorem{cor}[thm]{Corollary}

\theoremstyle{remark}
\newtheorem{rem}{Remark}
\newtheorem{example}{Example}

\def\-{\mbox{--}}

\begin{document}
\maketitle

\begin{abstract}
{\em Rainbow connection number} $rc(G)$ of a connected graph $G$ is the minimum number of colours needed to colour the edges of $G$, so that every pair of vertices is connected by at least one path in which no two edges are coloured the same. In this paper we show that for every connected graph $G$, with minimum degree at least $2$, the rainbow connection number is upper bounded by $\gamma_c(G) + 2$, where $\gamma_c(G)$ is the connected domination number of $G$. Bounds of the form $diameter(G) \leq rc(G) \leq diameter(G) + c$, $1 \leq c \leq 4$, for many special graph classes follow as easy corollaries from this result. This includes interval graphs, AT-free graphs, circular arc graphs, threshold graphs, and chain graphs all with minimum degree at least $2$ and connected. We also show that every bridge-less chordal graph $G$ has $rc(G) \leq 3.radius(G)$. In most of these cases, we also demonstrate the tightness of the bounds.

An extension of this idea to two-step dominating sets is used to show that for every connected graph on $n$ vertices with minimum degree $\delta$, the rainbow connection number is upper bounded by $3n/(\delta + 1) + 3$. This solves an open problem from \cite{schiermeyer2009rainbow}, improving the previously best known bound of $20n/\delta$ \cite{krivelevich2010rainbow}. Moreover, this bound is seen to be tight up to additive factors by a construction mentioned in \cite{caro2008rainbow}.
\end{abstract}

\noindent {\bf Keywords:} rainbow connectivity, rainbow colouring, connected dominating set, connected two-step dominating set, radius, minimum degree.

\section{Introduction}

{\em Edge colouring} of a graph is a function from its edge set to the set of natural numbers. A path in an edge coloured graph with no two edges sharing the same colour is called a {\em rainbow path}. An edge coloured graph is said to be {\em rainbow connected} if every pair of vertices is connected by at least one rainbow path. Such a colouring is called a {\em rainbow colouring} of the graph. The minimum number of colours required to rainbow colour a connected graph is called its {\em rainbow connection number}, denoted by $rc(G)$. For example, the rainbow connection number of a complete graph is $1$, that of a path is its length, that of an even cycle is its diameter, that of an odd cycle is one more than its diameter, and that of a star is its number of leaves. Note that disconnected graphs cannot be rainbow coloured and hence the rainbow connection number for them is left undefined. For a basic introduction to the topic, see Chapter $11$ in \cite{chartrand2008chromatic}.

The concept of rainbow colouring was introduced in \cite{chartrand2008rainbow}. Precise values of rainbow connection number for many special graphs like complete multi-partite graphs, Peterson graph and wheel graphs were also determined there. It was shown in \cite{chakraborty2009hardness} that computing the rainbow connection number of an arbitrary graph is NP-Hard. To rainbow colour a graph, it is enough to ensure that every edge of some spanning tree in the graph gets a distinct colour. Hence order of the graph minus one is an upper bound for rainbow connection number. There have been attempts to find better upper bounds for the same in terms of other graph parameters like connectivity, minimum degree etc. 

In the search towards good upper bounds for rainbow connection number, an idea that turned out to be successful more than once is a ``strengthened'' notion of connected $k$-step dominating set (Definition \ref{defn:domination} in Section \ref{sec:prelims}): a strengthening so that a rainbow colouring of the induced graph on such a set can be extended to the whole graph using a constant number of additional colours. Theorem $1.4$ in \cite{caro2008rainbow} was proved using a strengthened connected $1$-step dominating set and Theorem $1.1$ in \cite{krivelevich2010rainbow} was proved using a strengthened connected $2$-step dominating set. A closer examination revealed to us that the additional requirements imposed on the connected dominating sets in both those cases were far more restrictive than what was essential. This led us to the investigation of what is the weakest possible strengthening of a connected dominating set which can achieve the same. Since every edge incident on a pendant vertex will need a different colour, it is easy to see that such a dominating set should necessarily include all the pendant vertices in the graph. Quite surprisingly, it turns out that this obvious necessary condition is also sufficient! (Theorem \ref{thm:twoway} in Section \ref{sec:results}). For rainbow connection number of many special graph classes, the above result gives tight upper bounds which were otherwise difficult to obtain (Theorem \ref{thm:specialclass} in Section \ref{sec:results}). The farthest we could get with the idea was a curious theorem about chordal graphs (Theorem \ref{thm:chordal} in Section \ref{sec:results}).

A similar inquiry for the weakest strengthening a connected two-step dominating set (Theorem \ref{thm:twowaytwostep} in Section \ref{sec:results}) led us to the solution of an important open problem in this area regarding the optimal upper bound of rainbow connection number in terms of minimum degree. (See Theorem \ref{thm:rainmindeg} in Section \ref{sec:results} and the remarks therein). As an intermediate step in solving the above problem, we also discovered a tight upper bound on the size of a minimum connected two-step dominating set of a graph in terms of its minimum degree (Theorem \ref{thm:domsize} in Section \ref{sec:results}). To the best of our knowledge, this bound is not yet reported in literature. It may have applications beyond the realm of rainbow colouring. For instance, Theorem \ref{thm:domsize} immediately gives an upper bound on radius of every graph in terms of its minimum degree (Corollary \ref{cor:radius} in Section \ref{sec:results}) which marginally improves the one reported in \cite{erds1989radius}. 

\subsection{Preliminaries}
\label{sec:prelims}

See Table \ref{tab:notations} for the notations employed throughout the paper.

\begin{table}[t]
\begin{center}
\caption{Notations. $G$ is a graph, $v$ a vertex in $G$ and $S$ a subset of vertices in $G$. $k$ is a non-negative integer.}
\label{tab:notations}
\renewcommand{\tabcolsep}{0.5cm}
\renewcommand{\arraystretch}{1.3}
\begin{tabular}{ll}
& \\
\hline
$V(G)$ 			& Vertex set of $G$. \\
$E(G)$ 			& Edge set of $G$. \\
$|G|$ 			& Number of vertices in $G$ or order of $G$. \\
$\delta(G)$		& Minimum degree of $G$ \\
$pen(G)$ 		& Number of pendant vertices in $G$ \\
$rc(G)$ 		& Rainbow connection number of $G$ \\
$d(u,v)$		& Distance between vertices $u$ and $v$ \\
$ecc(v)$		& Eccentricity of $v$ \\
$diam(G)$		& Diameter of $G$ \\ 
$rad(G)$		& Radius of $G$ \\
$\gamma_c^k(G)$	& Connected $k$-step domination number of $G$ \\ 
$\gamma_c(G)$	& $\gamma^1_c(G)$, Connected domination number of $G$ \\ 
$N^k(S)$		& Set of all vertices at distance exactly $k$ from set $S$ \\
$N^k(v)$		& $N^k(\{v\})$ \\
$N(S)$			& $N^1(S)$, Neighbourhood of $S$ \\
$N(v)$			& $N^1(\{v\})$, Neighbourhood of $v$ \\
$G[S]$			& Induced subgraph of $G$ on $S$ \\
\hline
\end{tabular}
\end{center}
\end{table}

All graphs considered in this article are finite, simple and undirected. The {\em length} of a path is its number of edges.  An edge in a connected graph is called a {\em bridge}, if its removal disconnects the graph. A graph with no bridges is called a {\em bridge-less} graph.

\begin{defn}
Let $G$ be a connected graph. The {\em distance} between two vertices $u$ and $v$ in $G$, denoted by $d(u,v)$ is the length of a shortest path between them in $G$. The {\em eccentricity} of a vertex $v$ is $ecc(v) := \max_{x \in V(G)}{d(v, x)}$. The {\em diameter} of $G$ is $diam(G) := \max_{x \in V(G)}{ecc(x)}$. The {\em radius} of $G$ is $rad(G) := \min_{x \in V(G)}{ecc(x)}$. Distance between a vertex $v$ and a set $S \subseteq V(G)$ is $d(v, S) := \min_{x \in S}{d(v,x)}$. The {\em $k$-step open neighbourhood} of a set $S \subseteq V(G)$ is $N^k(S) := \{x \in V(G) | d(x, S) = k\}$, $k \in \{0, 1, 2, \cdots \}$. The {\em degree} of a vertex $v$ is $degree(v) := |N ^1(\{v\})|$. The minimum degree of $G$ is $\delta(G) := \min_{x \in V(G)}{degree(x)}$. A vertex is called {\em pendant} if its degree is $1$ and {\em isolated} if its degree is $0$. 
\end{defn}

\begin{defn}
\label{defn:domination}
Given a graph $G$, a set $D \subseteq V(G)$ is called a {\em $k$-step dominating set} of $G$, if every vertex in $G$ is at a distance at most $k$ from $D$. Further, if $D$ induces a connected sub-graph of $G$, it is called a {\em connected $k$-step dominating set} of $G$. The cardinality of a minimum connected $k$-step dominating set in $G$ is called its {\em connected $k$-step domination number} $\gamma^k_c(G)$. When $k = 1$, we may omit the qualifier ``$1$-step'' in the above names and the superscript $1$ in the notation.
\end{defn}

Note that connected $k$-step dominating sets exist only for connected graphs. Connected $k$-step domination number is left undefined otherwise.

\begin{defn}
An {\em intersection graph} of a family of sets $\mathcal{F}$, is a graph whose vertices can be mapped to sets in $\mathcal{F}$ such that there is an edge between two vertices in the graph if and only if the corresponding two sets in $\mathcal{F}$ have a non-empty intersection. An {\em interval graph} is an intersection graph of intervals on the real line. A {\em unit interval graph} is an intersection graph of unit length intervals on the real line. A {\em circular arc graph} is an intersection graph of arcs on a circle.  
\end{defn}

\begin{defn}
An independent triple of vertices $x$, $y$, $z$ in a graph $G$ is an {\em asteroidal triple $($AT$)$}, if between every pair of vertices in the triple, there is a path that does not contain any neighbour of the third. A graph without asteroidal triples is called an {\em AT-free graph.}
\end{defn}

\begin{defn}
A graph $G$ is a {\em threshold graph}, if there exists a weight function $w:V(G) \rightarrow \mathbb{R}$ and a real constant $t$ such that two vertices $u, v \in V(G)$ are adjacent if and only if $w(u) + w(v) \geq t$.
\end{defn}

\begin{defn}
A bipartite graph $G(A, B)$ is called a {\em chain graph} if the vertices of $A$ can be ordered as $A = (a_1, a_2, \cdots, a_k)$ such that $N(a_1) \subseteq N(a_2) \subseteq \cdots \subseteq N(a_k)$ \cite{yannakakis1982complexity}.
\end{defn}

\begin{defn}
A graph $G$ is called {\em chordal}, if there is no induced cycle of length greater than $3$.
\end{defn}

\section{Our Results}
\label{sec:results}

The main ideas in this paper are captured in Theorem \ref{thm:twoway}, Theorem \ref{thm:twowaytwostep} and Theorem \ref{thm:domsize}. The other results are consequences of them. Among the results, Theorem \ref{thm:rainmindeg} demands a special mention due to the prominence of the question it answers in the area of rainbow colouring. To state Theorem \ref{thm:twoway} in its full generality, we need to make one new definition.

\begin{defn}[Two-way dominating set]
A dominating set $D$ in a graph $G$ is called a {\em two-way dominating set} if every pendant vertex of $G$ is included in $D$. In addition, if $G[D]$ is connected, we call $D$ a {\em connected two-way dominating set}. 
\end{defn}

\begin{rem}
If $\delta(G) \geq 2$, then every (connected) dominating set in $G$ is a (connected) two-way dominating set. We use the name ``two-way domination'' since the definition implies that every vertex in $V(G) \backslash D$ has at least two edge disjoint paths to $D$.
\end{rem}

\begin{thm}
\label{thm:twoway}
If $D$ is a connected two-way dominating set in a graph $G$, then
$$ rc(G) \leq rc(G[D]) + 3. $$
\end{thm}

Proof is given in Section \ref{sec:twoway}

\begin{rem}
The reader may wonder why the pendant vertices had to be included in the dominating set $D$. Our strategy is to colour $G[D]$ first and then colour all the edges outside using a constant number (in this case $3$) of additional colours ensuring rainbow connectivity. Pendent vertices are always a bottleneck for rainbow colouring since no two pendant edges (edges incident on pendant vertices) can share the same colour. Hence the restriction.
\end{rem}

\begin{cor}
\label{cor:raindom}
For every connected graph $G$, with $\delta(G) \geq 2$,
$$ rc(G) \leq \gamma_c(G) + 2. $$
\end{cor}
\begin{proof}
This follows from Theorem \ref{thm:twoway} since (i) in this case, every connected dominating set in $G$ is a connected two-way dominating set and (ii) $rc(G[D]) \leq |D| - 1 = \gamma_c(G) - 1$ for a minimum connected dominating set $D$ in $G$.
\end{proof}

\begin{cor}
For every connected graph $G$, 
$$ rc(G) \leq \gamma_c(G) + pen(G) + 2.$$
\end{cor}
\begin{proof}
This follows from Theorem \ref{thm:twoway} since adding all the pendant vertices to a minimum connected dominating set gives a connected two-way dominating set of size at most $\gamma_c(G) + pen(G)$. 
\end{proof}

Diameter of a graph is a trivial lower bound for its rainbow connection number. Theorem \ref{thm:twoway} gives upper bounds which are only a small additive constant above the diameter for many special graph classes. 

\begin{thm}
\label{thm:specialclass}
Let $G$ be a connected graph with $\delta(G) \geq 2$. Then,
\begin{enumerate}[$(i)$]
\item if $G$ is an interval graph, $diam(G) \leq rc(G) \leq diam(G) + 1$, 
\item if $G$ is AT-free, $diam(G) \leq rc(G) \leq diam(G) + 3$,
\item if $G$ is a threshold graph, $diam(G) \leq rc(G) \leq 3$,
\item if $G$ is a chain graph, $diam(G) \leq rc(G) \leq 4$,
\item if $G$ is a circular arc graph, $diam(G) \leq rc(G) \leq diam(G) + 4$.
\end{enumerate}
Moreover, there exist interval graphs, threshold graphs and chain graphs with minimum degree at least $2$ and rainbow connection number equal to the corresponding upper bound above. There exists an AT-free graph $G$ with minimum degree at least $2$ and $rc(G) = diam(G) + 2$, which is $1$ less than the upper bound above. 
\end{thm}
\begin{rem}
The upper bounds follow from Theorem \ref{thm:twoway} since 
(i)  every interval graph $G$ which is not isomorphic to a complete graph has a dominating path of length at most $diam(G) - 2$,
(ii) every AT-free graph $G$ has a dominating path of length at most $diam(G)$,
(iii) a maximum weight vertex in a connected threshold graph $G$ is a dominating vertex,
(iv) every connected chain graph $G$ has a dominating edge, and
(v) every circular arc graph $G$, which is not an interval graph, has a dominating cycle of diameter at most $diam(G)$.
\end{rem}

Tight examples and proofs for non-trivial claims made in the above remark are given in Section \ref{sec:specialclass}. We could not find tight examples for AT-free and circular arc graphs. It may be interesting to see whether those two upper bounds can be improved.

Another interesting application of Theorem \ref{thm:twoway} is the following result on chordal graphs. It is curious since  chordal graphs, unlike interval graphs or AT-free graphs, can grow in more than two directions and hence they need not contain dominating paths in general.

\begin{thm}
\label{thm:chordal}
If $G$ is a bridge-less chordal graph, then $rc(G) \leq 3.rad(G)$. Moreover, there exists a bridge-less chordal graph with $rc(G) = 3.rad(G)$.
\end{thm}
Proof is given in Section \ref{sec:unitinterval}. The main idea is that we can induct on the radius of the graph and use Theorem \ref{thm:twoway} to prove the induction step.

Theorem \ref{thm:specialclass}$(i)$ gives $rc(G) \leq diam(G) + 1$ for every unit interval graph $G$. We have a stronger result, using a different approach.
\begin{thm}
\label{thm:unitinterval}
If $G$ is a unit interval graph such that $\delta(G) \geq 2$, then $rc(G) = diam(G)$.
\end{thm}
Proof is given in Section \ref{sec:unitinterval}

The extension of the idea of two-way domination to two-way two-step domination is what gives the remaining results. We need to make one more definition to state the next major theorem (Theorem \ref{thm:twowaytwostep}) in its full generality. 

\begin{defn}[Two-way two-step dominating set]
A (connected) two-step dominating set $D$ of vertices in a graph $G$ is called a {\em $($connected$)$ two-way two-step dominating set} if (i) every pendant vertex of $G$ is included in $D$ and (ii) every vertex in $N^2(D)$ has at least two neighbours in $N^1(D)$.
\end{defn}

\begin{rem}
As in the two-way ($1$-step) dominating set, here too every vertex $v \in V(G) \backslash D$ has two edge disjoint paths into $D$. Hence the adjective ``two-way''. It may be noted that, just like pendant edges, no two bridges in a graph can be coloured the same in any rainbow colouring. Hence the restriction of two-way domination is in some sense necessary to obtain colouring strategies which use only a constant number of extra colours outside the dominating set. 
\end{rem}

\begin{thm}
\label{thm:twowaytwostep}
If $D$ is a connected two-way two-step dominating set in a graph $G$, then
$$ rc(G) \leq rc(G[D]) + 6. $$
\end{thm}
Proof is given in Section \ref{sec:twowaytwostep}

\begin{thm}
\label{thm:domsize}
(i) Every connected graph $G$ of order $n$ and minimum degree $\delta$ has a connected two-step dominating set $D$ of size at most $\frac{3(n - |N^2(D)|)}{\delta + 1} - 2$. (ii) Every connected graph $G$ of order $n \geq 4$ and minimum degree $\delta$ has a connected two-way two-step dominating set $D'$ of size at most $\frac{3n}{\delta + 1} - 2$. Moreover, for every $\delta \geq 2$, there exist infinitely many connected graphs $G$ such that $\gamma^2_c(G) \geq \frac{3(n-2)}{\delta + 1} - 4$.
\end{thm}
Proof is given in Section \ref{sec:domsize}

It is easy to see that the radius of any connected graph is at most $k$ more than the radius of its $k$-step connected dominating set. Moreover, the radius of any graph $H$ is at most $|H|/2$. Hence the following corollary is also immediate. 

\begin{cor}
\label{cor:radius}
For every connected graph $G$ of order $n$ and minimum degree $\delta$, 
$$rad(G) \leq \frac{3n}{2(\delta + 1)} + 1.$$
\end{cor}

This bound marginally improves the one reported in \cite{erds1989radius}, namely $\frac{3(n-3)}{2(\delta + 1)} + 5$ and the proof is shorter. Note that we can similarly upper bound the diameter of $G$ by $\frac{3n}{\delta + 1} + 1$. But the corresponding bound reported in \cite{erds1989radius} is better, namely $\frac{3n}{\delta +1} -1$.

\begin{thm}
\label{thm:rainmindeg}
For every connected graph $G$ of order $n$ and minimum degree $\delta$,
$$rc(G) \leq \frac{3n}{\delta + 1} + 3.$$
Moreover, for every $\delta \geq 2$, there exist infinitely many graphs $G$ such that $rc(G) \geq \frac{3(n-2)}{\delta + 1} - 1$.
\end{thm}
\begin{proof}
Observe that the connected (two-way two-step dominating) set $D$ can be rainbow coloured using $|D| - 1$ colours by ensuring that every edge of some spanning tree gets distinct colours. So the upper bound follows immediately from Theorems \ref{thm:twowaytwostep} and \ref{thm:domsize}(ii). The family of tight examples is demonstrated in \cite{caro2008rainbow}.
\end{proof}

\begin{rem}
Theorem \ref{thm:rainmindeg} nearly settles the investigation for an optimal upper bound of rainbow connection number in terms of minimum degree which was initiated in \cite{caro2008rainbow}. There it was shown that, if $\delta(G) \geq 3$, then $rc(G) < 5n/6$. For general $\delta$, they had given two upper bounds viz., $(1 + o_{\delta}(1))n \ln(\delta) / \delta$ and $(4 \ln(\delta) + 3)n / \delta$. They had also shown a construction for a family of graphs with $diam(G) = \frac{3(n-2)} {\delta + 1} - 1$, leaving a gap of $\ln(\delta)$ factor between the bound and the construction.  They remarked that the problem of finding an optimum bound for $rc(G)$ in terms of $\delta$ is an intriguing problem and conjectured that for $\delta(G) \geq 3$, $rc(G) < 3n/4$. Schiermeyer proved the above conjecture and raised the question whether $rc(G) \leq 3n/(\delta + 1)$ for all values of $\delta$ \cite{schiermeyer2009rainbow}. If the answer is yes, then for graphs with linear minimum degree $\epsilon n$, the rainbow connection number is bounded by a constant. This was indeed shown to be the case in \cite{chakraborty2009hardness}. But their proof employed Szemer\'{e}di's Regularity Lemma and hence the bound was a tower function in $1/\epsilon$. This was considerably improved in \cite{krivelevich2010rainbow}, where it was shown that $rc(G) \leq 20n/ \delta$ for any connected graph. This is the best known bound for the problem till date. Theorem \ref{thm:rainmindeg} improves it and answers the question from \cite{schiermeyer2009rainbow} in affirmative but up to an additive constant of $3$. Moreover, this bound is seen to be tight up to additive factors by the construction mentioned in \cite{caro2008rainbow}.
\end{rem}

\section{Proofs}
\label{sec:proofs}

\subsection{Proof of Theorem \ref{thm:twoway}}
\label{sec:twoway}

{\em Statement.} 
If $D$ is a connected two-way dominating set in a graph $G$, then $rc(G) \leq rc(G[D]) + 3$.
\begin{proof}
We prove the theorem by demonstrating a rainbow colouring that uses at most $rc(G[D]) + 3$ colours. For $x \in N^1(D)$, its neighbours in $D$ will be called {\em foots} of $x$, and the corresponding edges will be called {\em legs} of $x$. Any rainbow path whose edge colours are contained in $\{1, 2, \cdots, k\}$ will be called a $k$-rainbow path.

Rainbow colour $G[D]$ using colours $\{1, 2, \cdots, k := rc(G[D])\}$. Let $H := G[V(G)\backslash D]$. Partition $V(H)$ into sets $X$, $Y$ and $Z$ as follows. $Z$ is the set of all isolated vertices of $H$. In every non-singleton connected component of $H$, choose a spanning tree. This gives a spanning forest on $V(H) \backslash Z$ with no isolated vertices. Choose $X$ and $Y$ as any one of the bipartitions defined by this forest. Colour every $X\-D$ edge with $k+1$, every $Y\-D$ edge with $k+2$ and every edge in $H$ with $k+3$. Since $D$ is a two-way dominating set, there are no pendant vertices outside $D$. Therefore, every vertex in $Z$ will have at least two legs. Colour one of them with $k+1$ and all the others with $k+2$.

We show that the above colouring is a rainbow colouring of $G$. For pairs in $D \times D$, there is already a $k$-rainbow path connecting them in $G[D]$. For a pair $(x, y) \in N^1(D) \times D$, join any leg of $x$ (coloured $k+1$ or $k+2$) with the $k$-rainbow path from the corresponding foot to $y$ in $G[D]$. For a pair $(x, y) \in (X \cup Z) \times (Y \cup Z)$ join a $k+1$ leg of $x$ and a $k+2$ leg of $y$ with a $k$-rainbow path between the corresponding foots in $G[D]$. For a pair in $(x, x') \in X \times X$, $x$ has a neighbour $y(x) \in Y$ from the spanning forest. Join the corresponding $x\-y(x)$ edge (coloured $k+3$) with the $y(x)\-x'$ $(k+2)$-rainbow path mentioned earlier. Similarly every pair $(y, y') \in Y \times Y$ is also rainbow connected. 
\end{proof}

\subsection{Proof of Theorem \ref{thm:twowaytwostep}}
\label{sec:twowaytwostep}

{\em Statement.} 
If $D$ is a connected two-way two-step dominating set in a graph $G$, then $rc(G) \leq rc(G[D]) + 6$.
\begin{proof}
We prove the theorem by demonstrating a rainbow colouring that uses at most $rc(G[D]) + 6$ colours. For $x \in N^k(D)$, its neighbours in $N^{k-1}(D)$, $k = 1,2$ will be called foots of $x$ and the corresponding edges will be called legs. Any rainbow path whose edge colours are contained in $\{1, 2, \cdots, k\}$ will be called a $k$-rainbow path.

Rainbow colour $G[D]$ using colours $\{1, 2, \cdots, k := rc(G[D])\}$. Construct a new graph $H$ on $N^1(D)$ with the edge set
\begin{eqnarray*}
E(H) 	&=& \{\{x,y\} | x, y \in N^1(D), \{x, y\} \in E(G) \textnormal{ or } \\
		&&	\exists z \in N^2(D) \textnormal{ such that } \{x, z\}, \{y, z\} \in E(G) \}.
\end{eqnarray*}

Recall that, in a two-way two-step dominating set $D$, there are no pendant vertices outside $D$ and every vertex in $N^2(D)$ has at least two neighbours in $N^1(D)$. Hence in the above graph $H$, the isolated vertices are only those which have all their neighbours (at least $2$) in $D$. Call their collection $Z$. Choose a spanning tree in every non-singleton connected component of $H$. This gives a spanning forest of $V(H) \backslash Z$ with no isolated vertices. Let $X$ and $Y$ be any bipartition defined by this forest. Colour every $X\-D$ edge with $(k+1)$ and every  $Y\-D$ edge with $(k+2)$. For every vertex in $Z$, colour one of its legs with $(k+1)$ and the remaining with $(k+2)$. Colour every edge of $G$ within $N^1(D)$ by $k+3$. Partition the vertices of $N^2(D)$ into $A$ and $B$ as follows. $A = \{x \in N^2(D) | N(x) \cap X \neq \emptyset \textnormal{ and } N(x) \cap Y \neq \emptyset \}$ and $B = N^2(D) \backslash A$. Colour every $A\-X$ edge with $(k+3)$ and every $A\-Y$ edge with $(k+4)$. First we claim that $G' := G[D \cup N^1(D) \cup A]$ is rainbow connected.

By following the same arguments as in proof of Theorem \ref{thm:twoway}, it can be easily seen that every pair in $D \times D$, is connected by a $k$-rainbow path and every pair in $N^1(D) \times D$ and $(X \cup Z) \times (Y \cup Z)$ is connected by a $(k+2)$-rainbow path. Notice that for every vertex $x \in X$, there exists $y(x) \in Y$ such that $x\-y(x)$ is an edge in the spanning forest. Vertices $x$ and $y(x)$ are connected either by a single $(k+3)$ edge or a $(k+3, k+4)$ path. Hence between any pair $(x, x') \in X \times X$, we can find a rainbow path by joining the $x\-y(x)$ path with the $y(x)\-x'$ $(k+2)$-rainbow path. Similarly any pair in $Y \times Y$ is also rainbow connected. Any pair $(a, a') \in A \times A$ can be rainbow connected by joining the $(k+3)$ leg of $a$ whose foot will be in $X$ and $(k+4)$ leg of $a'$ whose foot will be in $Y$ with the $(k+2)$-rainbow path between the two foots. Similarly we can connect any vertex $a \in A$ with any vertex in $x \in D \cup N^1(D)$ by using the $(k+3)$ leg of $a$ if $x \in Y$ and the $(k+4)$ leg of $a$ otherwise. Hence $G'$ is rainbow coloured using colours $1$ to $k+4$.  

Now only the vertices of $B$ remain. All of them have at least two neighbours in $G'$. Colour one edge to $G'$ with $k+5$ and all the other edges with $k+6$. It is easily seen that we now have a rainbow colouring of entire $G$.   
\end{proof}

\subsection{Proof of Theorem \ref{thm:domsize}}
\label{sec:domsize}

{\em Statement.} (i) Every connected graph $G$ of order $n$ and minimum degree $\delta$ has a connected two-step dominating set $D$ of size at most $\frac{3(n - |N^2(D)|)}{\delta + 1} - 2$. (ii) Every connected graph $G$ of order $n \geq 4$ and minimum degree $\delta$ has a connected two-way two-step dominating set $D'$ of size at most $\frac{3n}{\delta + 1} - 2$. Moreover, for every $\delta \geq 2$, there exist infinitely many connected graphs $G$ such that $\gamma^2_c(G) \geq \frac{3(n-2)}{\delta + 1} - 4$.
\begin{proof}
The case when $\delta \leq 1$ can be checked easily. So we assume $\delta \geq 2$ and execute the following two stage procedure.

\begin{enumerate}[{\bf Stage} $1$.]
\item
{
	$D = \{u\}$, for some $u \in V(G)$.\\ 
	While $N^3(D) \neq \emptyset$,\\
	\{\\
	\hspace*{5ex}Pick any $v \in N^3(D)$. Let $(v, x_2, x_1, x_0)$, $x_0 \in D$ be a shortest $v\-D$ path.\\
	\hspace*{5ex}$D = D \cup \{x_1, x_2, v \}$. \\ 
	\}
}
\item
{
	$D' = D$. \\
	While $\exists v \in N^2(D')$ such that $|N(v) \cap N^2(D')| \geq \delta - 1$,\\ 
	\{\\
	\hspace*{5ex}$D' = D' \cup \{x_1, v \}$  where $(v, x_1, x_0)$, $x_0 \in D'$ is a shortest $v\-D'$ path.\\
	\}
}
\end{enumerate}

Clearly $D$ remains connected after every iteration in Stage $1$. Since Stage $1$ ends only when $N^3(D) = \emptyset$, the final $D$ is a two step dominating set. Let $k_1$ be the number of iterations executed in Stage $1$. $|D \cup N^1(D)| \geq \delta + 1$ when Stage $1$ starts. Since a new vertex from $N^3(D)$ is added to $D$, $|D \cup N^1(D)|$ increases by at least $\delta + 1$ in every iteration. Therefore, when Stage $1$ ends, $k_1 + 1 \leq \frac{|D \cup N^1(D)|}{\delta + 1} = \frac{n - |N^2(D)|}{\delta + 1}$. Since $D$ starts as a singleton set and each iteration adds $3$ more vertices, $|D| = 3k_1 + 1 \leq  \frac{3(n - |N^2(D)|)}{\delta + 1} -2$. This proves Part (i) of the theorem.

$D'$ remains a connected two-step dominating set throughout Stage $2$. Stage $2$ ends only when every vertex $v \in N^2(D')$ has at most $\delta -2$ neighbours in $N^2(D')$. Hence at least two neighbours of $v$ are in $N^1(D')$. Moreover, since $\delta \geq 2$, there are no pendant vertices in $G$. So the final $D'$ is a connected two-way two-step dominating set. Let $k_2$ be the number of iterations executed in Stage $2$. Since we add to $D'$ a vertex who has at least $\delta - 1$ neighbours in $N^2(D')$, $|N^2(D')|$ reduces by at least $\delta$ in every iteration. Since we started with $|N^2(D)|$ vertices, $k_2 \leq \frac{|N^2(D)|}{\delta}$. Since we add $2$ vertices to $D'$ in each iteration, $|D'| = |D| + 2k_2 \leq \frac{3(n - |N^2(D)|)}{\delta + 1} - 2 + \frac{2|N^2(D)|}{\delta} \leq \frac{3n}{\delta + 1} -2$ for $\delta \geq 2$. This proves Part (ii). 

For every $\delta > 2$, construction for infinitely many graphs $G$ with $diam(G) = \frac{3(n-2)}{\delta + 1} - 1$ is reported in \cite{erds1989radius} and \cite{caro2008rainbow}. It is easy to see that for every graph $G$, $diam(G) \leq \gamma^2_c(G) + 3$. Hence $\gamma^2_c(G) \geq \frac{3(n-2)}{\delta + 1} -4$ for every graph in that family. 
\end{proof}

\subsection{Proof of Theorem \ref{thm:specialclass}}
\label{sec:specialclass}

Among the remarks made below Theorem \ref{thm:specialclass}, only (i), (ii) and (v) are non-trivial. Proof of (ii) can be found in \cite{corneil1997asteroidal}. We give proofs of (i) and (v) below.
\vspace{1ex}

\noindent {\em Statement} (i).
Every interval graph $G$ which is not isomorphic to a complete graph has a dominating path of length at most $diam(G) - 2$.
\begin{proof}
Consider an interval representation of $G = (V, E)$. For $u \in V$ let $l(u)$ and $r(u)$ represent the left end point and the right end point of $u$, respectively. Let $A = \min_{u \in V} r(u)$ and $B = \max_{u \in V} l(u)$. Let $a$ and $b$ be vertices such that $r(a) = A$ and $l(b) = B$. Let $S_1 = \{u \in V | l(u) \leq A\}$ and $S_2 = \{u \in V | r(u) \geq b\}$. Clearly $S_1$ and $S_2$ induce cliques in $G$ and thus we can assume that for each $u \in S_1$, $l(u) = A$ and for each $u \in S_2$, $r(u) = B$. Thus the intervals corresponding to $a$ and $b$ are point intervals. Since $G$ is not a complete graph, $A\neq B$. Moreover, since $G$ is connected, there exists $u \in S_1$ and $v \in S_2$ such that $u, v$ are not point intervals. Let $x \in S_1$ and $y \in S_2$ be two not necessarily distinct vertices such that the distance from $x$ to $y$ is minimum among all such pairs. Clearly $x \neq a$, $y \neq b$ and the shortest path $P$ between $x$ and $y$ is a dominating path in $G$. Moreover, since $a$ and $b$ are point intervals, $d(a, b) \geq d(x, y) + 2$. Hence length of $P$ is at most $d(a,b) -2 \leq diam(G) - 2$ as required.
\end{proof}

\noindent {\em Statement} (v).
Every circular arc graph $G$, which is not an interval graph, has a dominating cycle of diameter at most $diam(G)$.
\begin{proof}
Let $\mathcal{C}$ denote the circle in the circular arc representation of $G$. We will use the same symbol to denote a vertex of $G$ and its corresponding arc if there is no chance of confusion. Let $C_k =(c_1, c_2, \cdots, c_k)$ be a minimum collection of arcs that cover $\mathcal{C}$. It is easy to see that $C_k$ is a dominating cycle of $G$. We claim that $diam(C_k) \leq diam(G)$. 

For contradiction, let us assume $diam(C_k) > diam(G)$. Hence there exists $c_i, c_j \in V(C_k), \; i < j$, such that their distance in $G$ is less than their distance in $C_k$. Let $\mathcal{S}_1$ and $\mathcal{S}_2$ denote the two disjoint segments of $\mathcal{C} \backslash (c_i \cup c_j)$. Let $(c_i = x_0, x_1, \cdots, x_l, x_{l+1} = c_j)$ be a shortest $c_i\-c_j$ path in $G$. The set of arcs $X = (x_1, x_2, \cdots, x_l)$ will surely cover at least one of $\mathcal{S}_1$ or $\mathcal{S}_2$. Let $R := (c_{i+1}, c_{i+2}, \cdots c_{j-1})$ and $L := C_k \- (R \cup \{c_i, c_j\})$. Since arcs corresponding to $C_k$ cover the circle, the arcs corresponding one of them, say $L$  will cover the $\mathcal{S}_i$ not covered by $X$. By assumption $|L|, |R| > |X|= l$. So we can get a smaller collection of arcs covering $\mathcal{C}$ by replacing $R$ with $X$ in $C_k$ contradicting the minimality of $C_k$.
\end{proof}

\subsubsection*{Tight Examples}

We give examples to show that the upper bounds in $(i)$, $(iii)$ and $(iv)$ of Theorem \ref{thm:specialclass} are tight. We also give an example to show that the upper bound in $(ii)$ is nearly tight.

\begin{figure}[ht]
\begin{center}
\psset{xunit=1.5cm,yunit=1.5cm}
\begin{pspicture}(-1,-1)(7,1)
	\psset{labelsep=5pt,linewidth=0.5pt}
	\psdots[dotstyle=o, dotsize=5pt,fillstyle=solid, fillcolor=black]
		(-0.44, 0.88)(0.44, 0.88)(-0.44,-0.88)(0.44,-0.88)
		(0,0)(1,0)(2,0)(4,0)(5,0)(6,0)
		(5.56, 0.88)(6.44, 0.88)(5.56,-0.88)(6.44,-0.88)
	
	\psline[linecolor=black]{-}(0,0)(-0.44, 0.88)(0.44, 0.88)(0,0)(-0.44, -0.88)(0.44,-0.88)(0,0)(2.5,0)
	\psline[linecolor=black]{-}(6,0)(5.56, 0.88)(6.44, 0.88)(6,0)(5.56, -0.88)(6.44,-0.88)(6,0)(3.5,0)
	\psline[linecolor=black,linestyle=dashed](2.5,0)(3.5,0)
	
	\uput[l](0,0){$x_1$}
	\uput[d](1,0){$x_2$}
	\uput[d](2,0){$x_3$}
	
	\uput[d](4,0){$x_{d-3}$}
	\uput[d](5,0){$x_{d-2}$}
	\uput[r](6,0){$x_{d-1}$}

	\uput[u](-0.44,0.88){$y_{1}$}
	\uput[u](0.44,0.88){$y_{2}$}
	\uput[d](-0.44,-0.88){$y_{3}$}
	\uput[d](0.44,-0.88){$y_{4}$}

	\uput[u](5.56,0.88){$z_{1}$}
	\uput[u](6.44,0.88){$z_{2}$}
	\uput[d](5.56,-0.88){$z_{3}$}
	\uput[d](6.44,-0.88){$z_{4}$}

\end{pspicture}
\end{center}
\caption{Example of an interval graph with minimum degree $2$, diameter $d$ and rainbow connection number $d+1$.}
\label{fig:twinfan}
\end{figure}
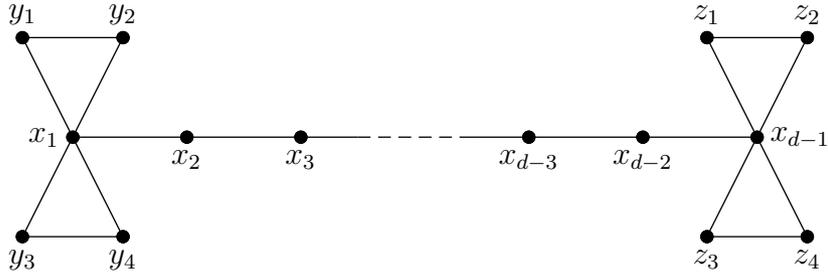

\begin{example}[An interval graph $G$ with $\delta(G) \geq 2$ and $rc(G) \geq d + 1$ for any given diameter $d$]
\label{ex:twinfan}
Consider the graph in Figure \ref{fig:twinfan}. It is an interval graph with minimum degree $2$ and diameter $d$. We claim that it cannot be rainbow coloured using $d$ colours. 

Let $Y = \{y_1, y_2, y_3, y_4\}$ and $Z = \{z_1, z_2, z_3, z_4\}$. Every pair $(y, z) \in Y \times Z$ is at a distance $d$ apart and they have only one $d$-length path between them. Hence every shortest $Y\-Z$ path should be rainbow coloured. So in any rainbow colouring which used only $d$ colours, every $Y\-x_1$ edge is forced to share the same colour. Hence there is no rainbow path between $y_1$ and $y_3$.
\end{example}

\begin{example}[An AT-free graph $G$ with $\delta(G) \geq 2$ and $rc(G) = diam(G) + 2$]
$K_{2,n}$, the complete bipartite graph with $2$ vertices in one part and $n$ in the other, is an AT-free graph with minimum degree $2$ and diameter $2$. For $n \geq 10$, its rainbow connection number is known to be $4$ (Theorem $2.6$ in \cite{chartrand2008rainbow}).
\end{example}

\begin{figure}[ht]
\begin{center}
\psset{xunit=1.5cm,yunit=1.5cm}
\begin{pspicture}(-0.5,-0.5)(7.5,1.5)
	\psset{labelsep=5pt,linewidth=0.5pt}
	\psdots[dotstyle=o, dotsize=5pt,fillstyle=solid, fillcolor=black]
        (3,0)(4,0)
		(0,1)(1,1)(2,1)(5,1)(6,1)(7,1)
	\psdots[dotstyle=o, dotsize=3pt]
		(3,1)(3.5,1)(4,1)
	
	\psline[linecolor=black]{-}(3,0)(0,1)(4,0)(1,1)(3,0)(2,1)(4,0)(5,1)(3,0)(6,1)(4,0)(7,1)(3,0)(4,0)

	\uput[d](3.5,0){$e$}
	
	\uput[d](3,0){$y_1$}
	\uput[d](4,0){$y_2$}

	\uput[u](0,1){$x_1$}
	\uput[u](1,1){$x_2$}
	\uput[u](2,1){$x_3$}
	
	\uput[u](5,1){$x_{n-3}$}
	\uput[u](6,1){$x_{n-2}$}
	\uput[u](7,1){$x_{n-1}$}

\end{pspicture}
\end{center}
\caption{Example of a threshold graph with minimum degree $2$ and rainbow connection number $3$. Also an example of a bridge-less chordal graph with $rc(G) = 3.rad(G)$.}
\label{fig:k2nplus}
\end{figure}
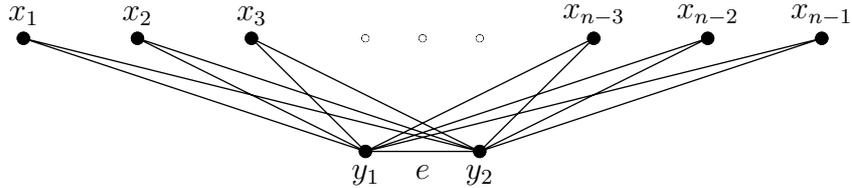
\begin{example}[A threshold graph $G$ with $\delta(G) \geq 2$ and $rc(G) \geq 3$]
\label{ex:threshold}
Consider the graph $G$ in Figure \ref{fig:k2nplus} which can be obtained by adding an edge $e$ between the two vertices in the smaller part of $K_{2,n-1}$, $n \geq 10$. It is easily seen to be a threshold graph (Two dominating vertices, $y_1$ and $y_2$, can be given a weight $1$, others a weight $0$ and threshold being $1$). For contradiction let us assume that $G$ can be coloured using $2$ colours. Subdividing $e$ gives a $K_{2,n}$. It is easy to see that by retaining the colour of $e$ to one of the two new edges and giving a third colour to the other is a rainbow colouring of $K_{2,n}$. This is a contradiction to the fact that $rc(K_{2,n}) = 4$ for $n \geq 10$.  (Theorem $2.6$ in \cite{chartrand2008rainbow}). 
\end{example}

\begin{example}[A chain graph $G$ with $\delta(G) \geq 2$ and $rc(G) \geq 4$]
$K_{2,n}$ is a chain graph with minimum degree $2$ and diameter $2$. For $n \geq 10$, it is known to have a rainbow connection number of $4$ (Theorem $2.6$ in \cite{chartrand2008rainbow}).  
\end{example}

\subsection{Proof of Theorem \ref{thm:chordal}}
\label{sec:chordal}

\begin{lem}
\label{lem:bridgelesschordal}
If $v$ is a vertex of eccentricity $r$ in a bridge-less chordal graph $G$, then $G[\bigcup_{i=0}^l{N^i(v)}]$ is a bridge-less chordal graph for all $l \in \{0, 1, \cdots, r\}$. 
\end{lem}
\begin{proof}
It is enough to show that $G' = G[\bigcup_{i=0}^{r-1}{N^i(v)}]$ is a bridge-less chordal graph. The general result will follow by repeated application of the above. Every induced subgraph of a chordal graph is also chordal. Hence it suffices to show that $G'$ is bridge-less.

For contradiction, let us assume that $b = (x,y) \in E(G')$ is a bridge of $G'$. Consider a BFS tree $T'$ of $G'$ rooted on $v$. 
Since $b$ is a bridge, $b \in E(T')$ (else $G'\backslash b$ will be connected). Without loss of generality let $x \in N^{i-1}(v)$ and $y \in N^{i}(v)$, $i \leq r-1$. Since $G$ is bridge-less by assumption, there exists a path from $x$ to $y$ in $G \backslash b$. Consider a shortest such path $P$. Since $P$ is a shortest path, $P \cup b$ is an induced cycle in $G$. Since $G' \backslash b$ is disconnected, $P$ has to contain at least one vertex $z$ from $N^r(v)$. Further, since $x \in N^{i-1}(v), \; i-1 \leq r-2$ cannot be adjacent to $z$, $P$ should contain at least one more vertex $w$ from $G'$. Hence $P \cup \{b\}$ is an induced cycle of length at least $4$ in $G$ which contradicts the assumption that $G$ is chordal.
\end{proof}

With the above lemma, now we can easily give the proof of Theorem \ref{thm:chordal}
\vspace{1ex}

\noindent {\em Statement.}
If $G$ is a bridge-less chordal graph, then $rc(G) \leq 3.rad(G)$. Moreover, there exists a bridge-less chordal graph with $rc(G) = 3.rad(G)$.
\begin{proof}
We will prove the statement by an induction on radius. Any graph with radius zero is a singleton vertex which can be rainbow coloured using zero colours. Hence the statement is true for radius zero. Let the statement be true up till a radius of $r -1$. 

Now, let $G$ be any bridge-less chordal graph with radius $r$. Let $v$ be a central vertex of $G$, i.e., a vertex with eccentricity $r$. By Lemma \ref{lem:bridgelesschordal}, $G' =  G[\bigcup_{i=0}^{r-1}{N^i(v)}]$ is also a bridge-less chordal graph and its radius is at most $r-1$. Hence by induction hypothesis $rc(G') \leq 3(r-1)$. Since minimum degree is at least two for any bridge-less graph, $V(G')$ is a connected two-way dominating set for $G$. Hence by Theorem \ref{thm:twoway}, $rc(G) \leq rc(G') + 3 \leq 3r$. Thus the statement is true for all values of radius.

Consider the graph $G$ in Figure \ref{fig:k2nplus}. It is a bridge-less chordal graph with radius $1$ and rainbow connection number is $3$. (See the argument under Example \ref{ex:threshold} in Section \ref{sec:specialclass}.)


\end{proof}

\subsection{Proof of Theorem \ref{thm:unitinterval}}
\label{sec:unitinterval}
{\em Statement.} 
If $G$ is a unit interval graph such that $\delta(G) \geq 2$, then $rc(G) = diam(G)$.
\begin{proof}
Let $G$ be a unit interval graph with $\delta(G) \geq 2$. Consider a unit interval representation of $G$. For $u \in V(G)$, let $l(u)$ and $r(u)$ represent the left and right end points of $u$ respectively. Let $x$ and $y$ be the vertices corresponding to the intervals with leftmost left end and rightmost right end respectively. Consider a shortest path $P$ between $x$ and $y$, say $P = (x=x_1, x_2, \ldots, x_k=y)$. Clearly $k \leq diam(G) + 1$ and $P$ is a dominating path in $G$. Let $S_i = \{u \in V(G) | l(u) \leq r(x_i) \leq r(u)\}$, $i = 1, 2, \cdots, k-1$. It is easily seen that each $S_i$ induces a clique in $G$. Let $H$ be a subgraph of $G$ with $V(H) = \bigcup_{i=1}^{k-1}{S_i}$ and $E(H) = \bigcup_{i=1}^{k-1}E(G[S_i])$. Since $G$ is a unit interval graph, $H$ contains all edges incident on $P$ (including those in $P$). Hence $H$ is a spanning subgraph of $G$. Colour every edge in $E(G[S_i]) \backslash \bigcup_{j=1}^{i-1}{E(G[S_j])}$ with colour $i$ for $i = 1, 2, \cdots, k-1$. This colours every edge of $H$ using at most $k-1$ colours. Colour the remaining edges of $G$ using colour $1$. We claim that this is a rainbow colouring of $G$.

For any pair of vertices, $u$ and $v$ in $V(G)$, consider any shortest path $R$ between them in $H$. Clearly $R$ does not contain more than one edge from a single clique. In the above colouring, two edges of $H$ will get the same colour only if they belong to the same clique. Hence $R$ is a rainbow path. So $rc(G) \leq k-1 \leq diam(G)$. Since diameter is a lower bound for rainbow connection number, $rc(G) = diam(G)$. 
\end{proof}

\bibliographystyle{apalike}
\bibliography{rainbow}
\end{document}